\documentclass[12pt]{article}
\usepackage{amsmath,amsfonts,amssymb,amsthm,amscd}
\title{Remarks on $NIP$  in a model}
\date{\today}
\author{Karim Khanaki\thanks{Partially supported by IPM grants 93030032 and 93030059}\\Arak University of Technology \and   Anand Pillay\thanks{Partially supported by NSF grants DMS-1360702 and DMS 1665035}\\University of Notre Dame}

\newtheorem{Theorem}{Theorem}[section]
\newtheorem{Proposition}[Theorem]{Proposition}
\newtheorem{Definition}[Theorem]{Definition}
\newtheorem{Remark}[Theorem]{Remark}
\newtheorem{Lemma}[Theorem]{Lemma}
\newtheorem{Corollary}[Theorem]{Corollary}
\newtheorem{Fact}[Theorem]{Fact}
\newtheorem{Example}[Theorem]{Example}
\newtheorem{Question}[Theorem]{Question}

\newcommand{\R}{\mathbb R}

\newcommand{\N}{\mathbb N}

\begin{document}
\maketitle

\begin{abstract} We define the notion $\phi(x,y)$ has $NIP$ in $A$, where $A$ is a subset of a model,  and give some equivalences by translating results from \cite{BFT}.
Using additional material from \cite{Simon} we discuss the number of coheirs when $A$ is not necessarily countable. We also revisit the notion ``$\phi(x,y)$ has $NOP$ in a model $M$" from \cite{Pillay-Grothendieck}.

\end{abstract}

\section{Introduction}
This paper is a kind of companion-piece to \cite{Pillay-Grothendieck}, although here  we are mainly concerned  with  direct translations of theorems from \cite{BFT}  into the (classical) model theory context.  The main results  are Corollary 2.2 on equivalences of ``$\phi(x,y)$ has $NIP$ in $A$",  Proposition 2.3 on number of coheirs when $A$ is not necessarily  countable, and Lemma 2.6  showing that the definition of ``$\phi(x,y)$ has not the order property in $M$" has an equivalent formulation compatible with the $NIP$ definitions in the current paper.

Our model theory notation is standard, and texts such as \cite{Poizat}, \cite{Pillay} will be sufficient background for the model theory part of the paper.
$IP$ stands for the independence property, and $NIP$ for not the independence property.

\begin{Definition} Let $T$ be a complete $L$-theory, $\phi(x,y)$ an $L$-formula, and $M$ a model of $T$.
\newline
(i) A  set $\{a_{\alpha}:\alpha < \kappa\}$ of $l(x)$-tuples from $M$ is said to be an $IP$-witness for $\phi(x,y)$ if  for all finite disjoint subsets $I, J$ of $\kappa$, $M\models \exists y(\bigwedge_{\alpha\in I} \phi(a_{\alpha},y) \wedge \bigwedge_{\beta\in J}\neg\phi(a_{\beta},y))$.
\newline
(ii) Let $A$ be a set of $l(x)$-tuples from $M$. Then $\phi(x,y)$ has $IP$ in $A$ if there is a countably infinite sequence $(a_{i}:i<\omega)$ of elements of $A$ which is an $IP$-witness for $\phi(x,y)$.
\newline
(iii) Let $A$ be a  set of $l(x)$-tuples in $M$. We say that $\phi(x,y)$ has $NIP$ in $A$ if it does not have $IP$ in $A$.
\newline
(iv) $\phi(x,y)$ has $NIP$ in $M$ if it has $NIP$ in the set of $l(x)$-tuples from $M$.
\end{Definition}

\begin{Remark} (i) $\phi$ has $NIP$ for the theory $T$ iff it has $NIP$ in every model $M$ of $T$ iff it has $NIP$ in some model $M$ of $T$ in which all types over the empty set in countably many variables are realised.
\newline
(ii)  If $\phi(x,y)$ has $IP$ in some model $M$ of $T$, then there arbitrarily long $IP$-witnesses for $\phi$ (of course in different models).
\newline
(iii) Let $(a_{\alpha}:\alpha < \kappa)$ be a collection of $l(x)$-tuples from $M$ and let $M^{*}$ be a saturated elementary extension of $M$ (i.e. $|M|^{+}$-saturated).  Then  $(a_{\alpha}:\alpha < \kappa)$ is an $IP$-witness for $\phi(x,y)$ iff there are $b_{I}$ in $M^{*}$ for each $I\subseteq \kappa$ such that $M^{*}\models \phi(a_{\alpha},b_{I})$ iff $\alpha\in I$, for all $\alpha<\kappa$, $I\subseteq \kappa$.
\end{Remark}

Given the $L$-formula $\phi(x,y)$, $\phi^{opp}(y,x)$ is the formula $\phi(x,y)$.

\begin{Example} Let $M$ be the structure with sorts $P = {\omega}$, $Q =$ finite subsets of $\omega$, and $R\subset P\times Q$ the membership relation.
Then the formula $R(x,y)$ has $IP$ in $M$ whereas $R^{opp}(y,x)$ has $NIP$ in $M$.
\end{Example}

In \cite{Pillay-Grothendieck},  the notion ``$\phi(x,y)$ has the order property (OP) in a model $M$" appeared, and we will discuss in 2.3 the compatibilities with Definition 1.1.

\vspace{5mm} \noindent It has been known for a long time that the
$NIP$ notion arose independently in model theory (\cite{Shelah})
and learning theory \cite{VC}.  More recently it was noticed by
several people (for example \cite{Chernikov-Simon},
\cite{Ibarlucia} and \cite{Khanaki}) that the notion also
appeared independently  in the context of function spaces
\cite{BFT}. The latter paper \cite{BFT} was at a fairly high
level of generality due to trying to find a common context for
functions on compact (Hausdorff) spaces and functions on Polish
spaces. The compact space case suffices for our purposes.

The current paper is partly expository, and has thematic overlap
with \cite{Simon}, \cite{Khanaki}, \cite{Ibarlucia}.  The notion
``$NIP$ of $\phi(x,y)$  in a model" is mentioned in
\cite{Khanaki}.  But \cite{Simon} and \cite{Ibarlucia} deal with
``$NIP$ of  $\phi(x,y)$  in a theory"  (in the continuous
framework in the latter paper). In the current paper we work in
classical ($\{0,1\}$-valued) model theory, although results such
as Corollary 2.2 are valid in the continuous logic framework. We
should also mention that in the continuous framework the formula
$|| x+y||$ in the language of Banach spaces is $NIP$ in
Tsirelson's space $M_{\cal T}$.  This was shown in
\cite{Khanaki-Banach}. It seems to be open whether the same
formula has $NIP$ in the theory of the Tsirelson space.

\subsection{Topology}
The data consists of a compact space $X$ and a collection $A$ of real valued functions on $X$.  The topology will be induced by the pointwise convergence topology on the ambient space $\R^{X}$. Let $B$ be some collection of real valued functions on $X$, containing $A$.

$A$ is said to be {\em relatively compact}  in $B$ if the closure $cl_{B}(A)$ of $A$ in $B$ is compact. Note that in this case $cl_{B}(A)$ is closed (and compact) in the space $\R^{X}$,  so in particular it implies that the closure of $A$ in $\R^{X}$ is contained in $B$.

$A$ is said to be {\em relatively countably compact} in $B$ if any sequence
\newline
$(f_{i}:i<\omega)$ has a cluster point (or accumulation point) in $B$.  Remember such a cluster point is by definition a point $f\in B$ such that any open neigbourhood of $f$ in $B$ contains infinitely many $f_{i}$.

A basic fact is that if $A$ is relatively compact in $B$ then it is relatively countably compact in $B$.

$A$ is said to be {\em relatively sequentially compact}  in $\R^{X}$, if any sequence from $A$ has a convergent (in $\R^{X}$) subsequence.

A subspace $B$ of $R^{X}$ is said to be {\em angelic} if (i) every relatively countably compact subset of $B$ is relatively compact (in $B$), and (ii) if $B_{0}\subset B$ is relatively compact
in $B$ then its closure is the set of limits of sequences from $B_{0}$.  (See \cite{BFT}.)

In the applications, typically $A$ will be a subset of $C(X)$,  the set of continuous functions on $X$.

\subsection{Model theory translation}

We fix an $L$-formula $\phi(x,y)$, and $L$-structure $M$ and a set $A$ of $l(x)$-tuples from $M$. We also let $M^{*}$ be an $|M|^{+}$-saturated elementary extension of $M$.
Let $X$ be the space $S_{\phi^{opp}}(A)$ of complete $\phi^{opp}$-types on $A$, namely  the (Stone) space of ultrafilters on the Boolean algebra generated by formulas $\phi(a,y)$ (or equivalently definable sets $\phi(a,y)(M)$) for $a\in A$.

Note that each formula $\phi(a,y)$ for $a\in A$ defines a  function on $X$, which takes $q\in X$ to $1$ if $\phi(a,y)\in q$ and to $0$ if $\phi(a,y)\notin q$.  We often write this value as $\phi(a,q)$, which note is also the (truth)value of $\phi(a,b)$ in $M^{*}$ for some/any $b$ realizing $q$.   We sometimes identify $A$ with this collection of (continuous) functions on $X$.

$S_{\phi}(M^{*})$ denotes the space of complete $\phi(x,y)$ types over $M^{*}$, namely ultrafilters on the Boolean algebra of formulas (definable sets) generated by $\phi(x,b)$ for $b\in M^{*}$.  Again given $p(x)\in S_{\phi}(M^{*})$ and $b\in M^{*}$ we can write $\phi(p,b)$ for the value of $\phi(x,b)$ at $p$.
$p(x)\in S_{\phi}(M^{*})$ is by definition {\em finitely satisfiable in $A$} if whenever $\chi(x)\in p$ then there is $a\in A$ such that $M^{*}\models \chi(a)$.
Note  that if $p(x)\in S_{\phi}(M^{*})$ is finitely satisfiable in $A$, and $b\in M^{*}$ then the value of $\phi(p,b)$ depends only on $tp_{\phi^{opp}}(b/A)$, so we may write this value as $\phi(p,q)$ where $q= tp_{\phi^{opp}}(b/A)\in X$.

The only additional thing we need to remark  on is the following (Remark 2.1 of \cite{Pillay-Grothendieck}:

\begin{Fact} Let $f\in 2^{X}$. Then $f\in cl(A)$ iff there is $p(x)\in S_{\phi}(M^{*})$, such that $f(q) = \phi(p,q)$ for all $q\in X$.
\end{Fact}

\section{Results}

\subsection{Theorems from  \cite{BFT} and the model theoretic translation}

The following can be extracted from \cite{BFT} and we give the explanation below. It is convenient to include compactness of the closure (in $\R^{X}$) of the given set $A$ of functions, in the assumptions.
\begin{Proposition}  Let $X$ be a compact Hausdorff space, and $A\subseteq C(X)$, the space of continuous real valued functions on $X$.
Assume $A$ to be uniformly bounded.
\newline
(a) The following are equivalent:
\newline
(i) Let $\alpha < \beta$, and $(f_{i}: i\in \N)$ a sequence from $A$. Then there is $I\subseteq \N$ such that $\{x\in X: f_{n}(x)\leq \alpha$ for all $n\in I$ and $f_{n}(x)\geq\alpha$ for all $n\notin I\}$ is empty.
\newline
(ii) $A$ is relatively sequentially compact in $\R^{X}$.

\vspace{2mm}
\noindent
(b) Suppose moreover that $A$ is countable, then each of (iii), (iv), (v) below is also equivalent to (i), (ii) above:
\newline
(iii) any $f\in \R^{X}$ in the closure of $A$ is a Borel (measurable) function,
\newline
(iv) each $f\in \R^{X}$ in the closure of $A$ is a pointwise limit of a sequence $(f_{i}:i\in\N)$ from $A$,
\newline
(v) $cl(A)$ has cardinality $< 2^{2^{\omega}}$.
\end{Proposition}
\noindent
{\em Explanation.}  (a)  is precisely the equivalence of (vi) and (ii) in Theorem 2F of \cite{BFT}. Note that this Theorem 2F  has weaker general assumptions: $X$ is an aribtrary Hausdorff space  and $A\subseteq C(X)$ is pointwise bounded.
\newline
(b) can be extracted from Theorem 4D of \cite{BFT}, as we describe now.  The  assumptions of Theorem 4D are again much more general, but it states an equivalence between eight conditions.  Theorem 4D (viii) is our (a)(i). Theorem 4D(vii) is our (b)(v).  Theorem 4D (vi) says that the closure $\bar A$ of $A$ in $\R^{X}$ is {\em angelic}.  (See section 1.1.)  As $\bar A$ is compact, we obtain (b)(iv).  As $A$ is countable, this immediately implies (b)(iii) that ${\bar A}$ consists of Borel functions.  But then $A$ is relatively compact in $B(X)$ so also relatively countably compact in $B(X)$, which is precisely Theorem 4D(v) of \cite{BFT}.

\vspace{5mm}
\noindent
Let $\phi(x,y)$, $M$, $A$, $M^{*}, X$ be as in Section 1.2. Then translating conditions (i) - (v) in Proposition 2.1 yields the following characterizations of $\phi(x,y)$ has $NIP$ in $A$:

\begin{Corollary}
(a) The following are equivalent:
\newline
(i) $\phi(x,y)$ has $NIP$ in $A$.
\newline
(ii) For any $(a_{i}:i<\omega)\subseteq A$ there is a subsequence $(a_{j_{i}}:i <\omega)$ such that for any $b\in M^{*}$ there is an eventual truth value of $(\phi(a_{j_{i}}, b): i< \omega)$.
\newline
(b)  Moreover if $A$ is countable, then the following are also equivalent to (i), (ii):
\newline
(iii) Any $p(x)\in S_{\phi}(M^{*})$ which is finitely satisfiable in $A$ is Borel definable over $A$, namely
 $\{q\in S_{\phi^{opp}}(A): \phi(x,b)\in p$ for some/any $b\in M^{*}$ realizing $q$\} is Borel.
\newline
(iv) For any $p(x)\in S_{\phi}(M^{*})$ which is finitely satisfiable in $A$, there is a sequence $(a_{i}:i<\omega)$ from $A$ such that for any $b\in M^{*}$, %
$\phi(p,b)$ is the eventual truth value of $(\phi(a_{i},b):i<\omega)$.
\newline
(v)  The number of $p(x)\in S_{\phi}(M^{*})$ which is finitely satisfiable in $A$ is $< 2^{2^{\omega}}$ (in fact easily seen to be $\leq 2^{\omega}$).

\end{Corollary}

\subsection{Number of finitely satisfiable global types}

Again we stick with the set up in Section 1.2. We make use of a theorem from \cite{Simon} in addition to Theorem 2F of \cite{BFT} to obtain:

\begin{Proposition} Suppose that $\phi(x,y)$ has $NIP$ in $A$ where $A$ has cardinality at most  $\kappa$. Then the cardinality of the set of $p(x)\in S_{\phi}(M^{*})$ such that $p$ is finitely satisfiable in $M$ is at most $2^{\kappa}$.
\end{Proposition}
\noindent
When $A$ is countable, this is given by Corollary 2.2,  (i) implies (v).

\vspace{2mm}
\noindent
Before giving the proof let us remark:
\begin{Remark} One can not expect a converse to Proposition 2.3. But note that we have the  partial converse: Suppose $A$ is a set of $l(x)$-tuples from $M^{*}$ and  $(a_{\alpha}:\alpha<\kappa)\subseteq A$ is an $IP$ witness for  $\phi(x,y)$ (see Definition 1.1(i)).   Then the cardinality of the set of $p(x)\in S_{\phi}(M^{*})$ which is finitely satisfiable in $A$ is at least $2^{2^{\kappa}}$.
\end{Remark}
\noindent
{\em Explanation of 2.4.}  This is the usual proof of many coheirs for formulas with the independence property. Namely for each subset $I$ of $\kappa$ let $b_{I}\in M^{*}$ be such that  $M^{*}\models \phi(a_{\alpha}, b_{I})$ iff $\alpha\in I$.
For  each ultrafilter $U$ on $\kappa$, let $\Sigma_{U}(x)$ be the collection of formulas $\{\phi(x,b_{I});I\in U\}\cup\{\neg\phi(x,b_{I}:I\notin U\}$.
Each $\Sigma_{U}(x)$ is finitely satisfiable in $A$, so  extends to some $p_{U}(x)\in S_{\phi}(M^{*})$ finitely satisfiable in $A$.  The $p_{U}$ are clearly mutually contradictory and there are $2^{2^{\kappa}}$ of them.

\vspace{5mm}
\noindent
{\em Proof of Proposition 2.3.}   The proof involves some additional ingredients which we explain along the way.
We are working again in the model theory context of $\phi(x,y)$, $A$, $M^{*}$, $X = S_{\phi^{opp}}(A)$, where we identify $A$ with the $A$ the collection of continuous functions from $X$ to $\{0,1\}$ given by formulas $\phi(a,y)$ for $a$ in $M$.
Following Definition 2.1 of \cite{Simon}, $B^{\S}(X)$ is the set of $\{0,1\}$-valued functions $f$ on $X$ such that ${\overline{f^{-1}(0)}}\cap {\overline{f^{-1}(1)}}$ has empty interior, and $B^{\S}_{r}(X)$ is the set of $\{0,1\}$-valued functions $f$ on $X$ such that $f|Y \in B^{\S}(Y)$ for every nonempty closed subset $Y$ of $X$.
Now Theorem 2F (vi) of \cite{BFT}  says that $\phi(x,y)$ has $NIP$ in $M$.  And  part (iii) of Theorem 2F says that for every ``Cech-complete" subset $Y$ of $X$, $\{f|Y: f\in A\}$ is relatively compact in $B^{\S}(Y)$. (We are here using the fact that $A$ consists of $\{0,1\}$-valued functions as well as the compatibility of Definition 1A of \cite{BFT} with Definition 2.1 of \cite{Simon} that we have above.) As every closed subset $Y$ of $X$ is Cech-complete (see 2A of \cite{BFT}),  we conclude that for every closed $Y\subseteq X$, $\{f|Y: f\in A\}$ is relatively compact in $B^{\S}(Y)$.  Now suppose $h\in cl(A)$ (in the space $2^{X}$).  So $h\in B^{\S}(X)$.  As we are working with the pointwise convergence topology, $h|Y \in cl(\{f|Y:f\in A\}$ for every closed $Y\subseteq X$.  But then $h|Y\in B^{\S}(Y)$. Hence by definition,
\newline
(*)  $h\in B^{\S}_{r}(X)$.

We now want to apply Theorem 2.3 of \cite{Simon}.  The assumption  that $X$ is a compact Hausdorff  $0$-dimensional space with a basis of size at most $\kappa$ is satisfied.
Let $P_{<\omega}(\kappa)$ be the set of finite subsets of $\kappa$ and ${\frak F}_{\kappa}$ the filter on $P_{<\omega}(\kappa)$ generated by sets $T_{j} = \{i\in P_{<\omega}(\kappa): i\supseteq j\}$.   $B_{1}(X)$ is then defined to be the set of functions $h:X\to \{0,1\}$ such that that there is a family $(f_{i}:i\in P_{<\omega}(\kappa))$ of continuous functions from $X$ to $\{0,1\}$ such that for every neighbourhood $U$ of $h$, $\{i\in P_{<\omega}(\kappa): f_{i}\in U\} \in {\frak F}_{\kappa}$.  Note that $h$ is determined by the family
$\{f_{i}:i\in P_{<\omega}(\kappa)\}$.    As the collection of continuous functions from $X$ to $\{0,1\}$ is of cardinality at most $\kappa$ it follows that the cardinality of $B_{1}(X)$ has cardinality at most $2^{\kappa}$.

Now  Theorem 2.3 of \cite{Simon} concludes from (*) that $h\in B_{1}(X)$.  So putting everything together we see that $|cl(A)| \leq 2^{\kappa}$.  By Fact 1.4, there are at most $2^{\kappa}$ $p(x)\in S_{\phi}(M^{*})$ which are finitely satisfiable in $A$.

\begin{Question} Under the assumptions of Proposition 2.3, can we say anything about the cardinality of the set of $p(x)\in S_{\phi}(M^{*})$ which are $Aut(M^{*}/A)$-invariant?

\end{Question}

\subsection{$NOP$ in a model}

In \cite{Pillay-Grothendieck} as well as implicitly in earler papers (e.g. \cite{Pillay-dimension}), we defined``$\phi(x,y)$ has not the order property ($NOP$) in a model $M$" to mean that there do not exist $a_{i}, b_{i}\in M$ for  $i<\omega$ such that $M\models \phi(a_{i},b_{j})$ iff $i\leq j$ for all $i,j$, or $M\models \phi(a_{i},b_{j})$ iff $i\geq j$ for all $i,j$.  This definition is not, on the face of it, the right analogue of ``$\phi(x,y)$ has $NIP$ in $M$".   However we point out quickly that it is {\em equivalent} to the right analogue.

\begin{Lemma} Let $\phi(x,y)$ be an $L$-formula, $M$ an $L$-structure, and $M^{*}$ a saturated elementary extension of $M$. Then the following are equivalent:
\newline
(i) there exist $a_{i}, b_{i}\in M$ for $i<\omega$ such that either $M\models \phi(a_{i},b_{j})$ iff $i\leq j$ for all $i,j$ or $M\models \phi(a_{i},b_{j})$ iff $i\geq j$ for all $i,j$.
\newline
(ii)  there exist $a_{i}\in M$ for $i<\omega$ and $b_{i}\in M^{*}$ for $i<\omega$ such that,  either $M^{*}\models \phi(a_{i},b_{j})$ iff $i\leq j$ for all $i,j$, or $M\models\phi(a_{i},b_{j})$ iff $i\geq j$ for all $i,j$.
\end{Lemma}
\begin{proof}
Clearly (i) implies (ii).
 Suppose now that (i) fails, namely that $\phi(x,y)$ does not have the order property in $M$.
Suppose, for a contradiction that (ii) holds. Without loss of generality, we have $a_{i}\in M$ and $b_{i}\in M^{*}$ such that $M\models \phi(a_{i},b_{j})$ iff $i\leq j$.  Now by Proposition 2.3 of \cite{Pillay-Grothendieck} and our assumption that (i) fails, there is a subsequence $(a_{j_{i}}:i<\omega)$ of $(a_{i}:i<\omega)$ and a finite Boolean combination$\psi(y)$ of instances $\phi(a,y)$ of $\phi(x,y)$ for $a\in M$,  such that for all $b\in M^{*}$, the truth  value of $\psi(b)$ equals the eventual truth value of $\phi(a_{j_{i}},b)$.  But note that $M^{*}\models \neg\psi(b_{i})$ for each $i<\omega$. So by compactness we can find $b\in M^{*}$ such that $M^{*}\models \neg\psi(b)$ but $M^{*}\models \phi(a_{j_{i}},b)$ for all $i<\omega$, a contradiction.

\end{proof}

\begin{Corollary} If $\phi(x,y)$ has $NOP$ in $M$ (i.e. negation of condition (i) in Lemma 2.6) then $\phi(x,y)$ has $NIP$ in $M$ in the sense of Definition 1.1.

\end{Corollary}

\end{document}